\documentclass[12pt]{article}

\usepackage[a4paper,top=3cm,bottom=2cm,left=2cm,right=2cm,marginparwidth=1.25cm]{geometry}
\usepackage{url}

\usepackage{graphicx}
\usepackage{amsmath}
\usepackage{amsthm}
\usepackage{amsfonts}
\usepackage{amssymb}
\usepackage{xcolor}
\usepackage{prettyref}
\usepackage{enumerate}
\usepackage{xspace}
\usepackage{bm}

\newtheorem{theorem}{Theorem}

\newtheorem{corollary}[theorem]{Corollary}

\theoremstyle{definition}

\newtheorem{remark}{Remark}

\newcommand {\myvec}[1] {{\mbox{\boldmath $#1$}}}

\ifx\eqref\undefined
\newcommand{\eqref}[1]{~(\ref{#1})}
\fi
\ifx\mod\undefined
\def\mod{\mathop{\rm mod}}
\fi

\usepackage{bm}

\newcommand{\spnorm}[1]{\left\Vert#1\right\Vert_{\mathsf{sp}}}

\def\argmin{\mathop{\rm argmin}}

\def\EE{\Expect}

\def\Var{\mathrm{Var}}

\def\PP{\mathbb{P}}

\def\eqdef{\triangleq}

\def\Risk{\mathsf{Risk}}

\newcommand{\bsym}[1]{\boldsymbol{#1}}

\newcommand{\reals}{\mathbb{R}}

\newcommand{\Expect}{\mathbb{E}}

\newcommand{\pth}[1]{\left( #1 \right)}
\newcommand{\qth}[1]{\left[ #1 \right]}
\newcommand{\sth}[1]{\left\{ #1 \right\}}
\newcommand{\abs}[1]{\left| #1 \right|}

\newcommand{\bs}[1]{\boldsymbol{#1}}

\definecolor{myblue}{rgb}{.8, .8, 1}
\definecolor{mathblue}{rgb}{0.2472, 0.24, 0.6} 
\definecolor{mathred}{rgb}{0.6, 0.24, 0.442893}
\definecolor{mathyellow}{rgb}{0.6, 0.547014, 0.24}

\newcommand{\sfT}{{\mathsf{T}}}

\newcommand{\calN}{{\mathcal{N}}}

\newcommand{\calS}{{\mathcal{S}}}

\newrefformat{eq}{(\ref{#1})}
\newrefformat{thm}{Theorem~\ref{#1}}
\newrefformat{th}{Theorem~\ref{#1}}
\newrefformat{chap}{Chapter~\ref{#1}}
\newrefformat{sec}{Section~\ref{#1}}
\newrefformat{algo}{Algorithm~\ref{#1}}
\newrefformat{fig}{Fig.~\ref{#1}}
\newrefformat{tab}{Table~\ref{#1}}
\newrefformat{rmk}{Remark~\ref{#1}}
\newrefformat{clm}{Claim~\ref{#1}}
\newrefformat{def}{Definition~\ref{#1}}
\newrefformat{cor}{Corollary~\ref{#1}}
\newrefformat{lmm}{Lemma~\ref{#1}}
\newrefformat{prop}{Proposition~\ref{#1}}
\newrefformat{pr}{Proposition~\ref{#1}}
\newrefformat{app}{Appendix~\ref{#1}}
\newrefformat{apx}{Appendix~\ref{#1}}
\newrefformat{ex}{Example~\ref{#1}}
\newrefformat{exer}{Exercise~\ref{#1}}
\newrefformat{soln}{Solution~\ref{#1}}
\newrefformat{ineq}{inequality~(\ref{#1})}
\newrefformat{goal}{Goal~\ref{#1}}
\newrefformat{pf}{Appendix~\ref{#1}}

\def\mA{{\bm{A}}}
\def\mB{{\bm{B}}}
\def\mC{{\bm{C}}}

\def\mI{{\bm{I}}}

\def\mQ{{\bm{Q}}}

\def\mX{{\bm{X}}}

\usepackage{amsmath,amsfonts,bm}

\usepackage[
            CJKbookmarks=true,
            bookmarksnumbered=true,
            bookmarksopen=true,
            colorlinks=true,
            citecolor=red,
            linkcolor=blue,
            anchorcolor=red,
            urlcolor=blue
            ]{hyperref}

\def\1{\bm{1}}

\def\ervz{{\textnormal{z}}}

\def\va{{\bm{a}}}
\def\vb{{\bm{b}}}
\def\vc{{\bm{c}}}

\def\vq{{\bm{q}}}

\def\vu{{\bm{u}}}
\def\vv{{\bm{v}}}

\def\vx{{\bm{x}}}
\def\vy{{\bm{y}}}
\def\vz{{\bm{z}}}

\def\mA{{\bm{A}}}
\def\mB{{\bm{B}}}
\def\mC{{\bm{C}}}

\def\mI{{\bm{I}}}

\def\mQ{{\bm{Q}}}

\def\mX{{\bm{X}}}

\DeclareMathAlphabet{\mathsfit}{\encodingdefault}{\sfdefault}{m}{sl}
\SetMathAlphabet{\mathsfit}{bold}{\encodingdefault}{\sfdefault}{bx}{n}

\def\sS{{\mathbb{S}}}

\usepackage{blindtext}
\usepackage{algorithm,algorithmic}
\usepackage{mathtools}

\providecommand{\keywords}[1]{\textbf{Keywords:} #1}

\begin{document}

\title{Support recovery with Projected Stochastic Gates: theory and application for linear models}


\author{Soham Jana, Henry Li, Yutaro Yamada, Ofir Lindenbaum \thanks{
			S.~Jana is with the Department of Operations Research and Financial Engineering, Princeton University, Princeton, NJ, USA, email: \url{soham.jana@princeton.edu}. Y.~Yamada is with the Department of Statistics and Data Science, Yale University, New Haven, CT, email: \url{yutaro.yamada@yale.edu}. H.~Li is with the Department of Applied Mathematics, Yale University, New Haven, CT, email: \url{henry.li@yale.edu}. O.~Lindenbaum is with the Faculty of Engineering, Bar-Ilan University, Ramat Gan, Israel, email: \url{ofir.lindenbaum@biu.ac.il}.}
			}

\maketitle  
\begin{abstract}
Consider the problem of simultaneous estimation and support recovery of the coefficient vector in a linear data model with additive Gaussian noise. We study the problem of estimating the model coefficients based on a recently proposed non-convex regularizer, namely the stochastic gates (STG) \cite{yamada20a}. We suggest a new projection-based algorithm for solving the STG regularized minimization problem, and prove convergence and support recovery guarantees of the STG-estimator for a range of random and non-random design matrix setups. Our new algorithm has been shown to outperform the existing STG algorithm and other classical estimators for support recovery in various real and synthetic data analyses.  \\

\noindent\keywords{Support recovery, sparsity, feature selection, consistency, non-convex penalty, LASSO, SCAD, orthogonal matching pursuit, stochastic gates, compressed sensing}

\end{abstract}






    
    

\section{Introduction}
We study the following question: given observations from a noisy linear model, how can we recover the positions of the non-zero entries of the unknown coefficient vector (also known as the {\it sparsity pattern} or {\it support-set} \cite{wainwright09})? The analysis of the above problem, which is often referred to as sparse recovery, has seen many uses across fields ranging from theoretical computer science to applied mathematics to digital signal processing. Examples include compressed sensing \cite{donoho2006compressed}, image denoising \cite{elad2006image,zhao2014hyperspectral}, manifold learning \cite{guo2010action,DUFS}, etc. 
One popular solution for sparse recovery is the Least Absolute Shrinkage and Selection Operator (LASSO) \cite{tibshirani1996regression,wainwright09}, which has several extensions, such as \cite{irls,irl1,isd}. Greedy methods, such as Orthogonal Matching Pursuit (OMP) \cite{omp}, Randomized OMP \cite{randomp} and their extensions \cite{elad2009plurality,stomp,cosamp}, are well studied in this context. Several authors \cite{simon2019mmse,lindenbaum2021randomly,lindenbaum2021refined} have recently demonstrated that introducing controlled noise in the optimization process can lead to improved model performance. 

Another interesting avenue of research on this topic involves the analysis of non-convex penalties that approximate the sparsity constraints in the objective function. Examples of such penalties include the well-known smoothly clipped absolute deviation (SCAD) \cite{scad,FFW09,XH09,PL09}. Even though estimators based on conventional convex penalties (e.g., LASSO) come with computational benefits, it is shown in many applications that non-convex penalties can yield significantly improved performances (see \cite{MFJH11,WCLQ18} for detailed discussions). In recent years many new non-convex penalties have surfaced in the study of sparse Neural Networks \cite{yamada20a,jang2016categorical,maddison2016concrete,louizos2017learning}. Estimators based on these new penalties show promising real and synthetic data analysis performances. However, analysis of their properties is sometimes difficult due to the complex nature of the models and penalties.

In our current work, we are interested in analyzing the non-convex-penalty-based estimator of the stochastic gates (STG) proposed in \cite{yamada20a}. The STG is an ``Embedded method" for support recovery that simultaneously estimates the support and coefficient values and hence works as an improvement over ``Filter methods" (which attempt to remove irrelevant features before learning a model) and ``Wrapper methods" (which recompute the model for each subset of features and, thus, become computationally expensive). The readers are directed to \cite[Section 1]{yamada20a} for a detailed list of references for classical support-recovery techniques. Using information-theoretic arguments, \cite[Section 4]{yamada20a} justified that solving the optimization problem for constrained subset selection is equivalent to optimizing the parameters of the Bernoulli gates for the features and argued why the non-convex penalty used for computing the STG estimates should achieve that objective. In addition, the objective function for computing the STG estimates also uses a noisy and continuous version of the discrete Bernoulli gates, with the hope of obtaining improved model performances as mentioned before. Indeed, the STG has been shown to produce competitive performances in various simulation studies. In the case of non-linear models, \cite{yamada20a} studied the Cox Proportional Hazard Model and Noisy binary XOR classification models, and showed that the STG-based estimators provide more accurate and stable performances compared to alternative methods based on Random Forests, Hard-Concrete penalty, the Lasso, and a deterministic non-convex counterpart of the STG. In the case of feature selection based on linear models with the squared error loss, they show that while the STG and Hard-Concrete-penalty based estimators perform better support recovery than the Lasso, the estimates provided by the STG have lower variances. This low variance is usually attributed to the lighter tail of the Gaussian-based penalty of the STG, compared to the heavy-tailed logistic-based penalty of the Hard-Concrete method. Despite these practical benefits, we note that there is scope for improving the existing STG algorithm in the linear setup, which can lead to improving the approximates of the STG estimator. The existing STG algorithm is mainly built for tackling a general class of data generating models and loss functions and relies fully on the gradient descent algorithm. However, borrowing ideas from the linear model theory, we show that in certain steps for estimating the coefficient vectors in the STG algorithm, with high probability, it is possible to provide closed-form optimal solutions. As expected, it has been shown later in our work that the new algorithm, termed the Projected-STG method, provides support recovery guarantees that are at least as good as the existing STG algorithm of \cite{yamada20a} and various other baseline methods in real and simulated data examples. Moreover, the performance of the Projected-STG is significantly better than that of its competitors when the signal-to-noise ratio is smaller. In addition, under benign conditions, we provide basic consistency guarantees for the STG algorithm which was previously missing from the literature.

The rest of the paper is organized as follows. In \prettyref{sec:problem} we define our notations and describe the optimization problem. In \prettyref{sec:algorithm} we describe the Projected-STG algorithm for linear models. \prettyref{sec:theory} contains some asymptotic results about the STG estimator in both random and fixed design matrix setups. Finally, we end our discussion with simulation studies in \prettyref{sec:simulation}.


\subsection{Problem Setting}
\label{sec:problem}
Let $\mX=[\vx_1,\dots,\vx_N]^\sfT$ be a set of $N$ input vectors of dimension $D$ each and suppose that data is generated via the linear model 
\begin{align}\label{eq:main-model}
	\vy=\mX\bm\beta^*+\bm\epsilon,\quad 
	\bm\epsilon\sim \calN(0,\sigma^2\mI_N),
\end{align}
where $\mI_N$ is an identity matrix of dimension $N$, $\calN(0,\sigma^2\mI_N)$ denotes a multivariate Gaussian distribution with mean $0$ and covariance $\sigma^2\mI_N$, and $\sigma$ is assumed to be known. For any vector $\vv$, let $\|\vv\|_0$ denote the sparsity (i.e., the number of nonzero entries) of $\vv$ and $\|\vv\|_2$ denotes the Euclidean norm of $\vv$. When it is known beforehand that the support-set of $\beta^*$ has at most $K$ elements the optimization problem with the squared error loss translates to finding
\begin{align}\label{eq:l0-problem}
	\argmin_{\bs\beta} \frac{\|\mX \bm\beta - \vy\|^2_2}N \; \text{such that $\|\bm\beta\|_0 \leq K$}.
\end{align}
Due to the intractability of the above problem (as it involves a discrete search over an exponential number of possible parameter vectors, see \cite{natarajan1995sparse,garey1979computers} for more details) it is general practice to solve instead some penalized objective that approximates the sparsity constrained problem. The idea of penalization via \textit{gates} uses separate random variables that work as filters on the parameters. Consider the re-parameterization of $\bm\beta$ via random variables $\vz=\left\{z_d\right\}_{d=1}^D$ which we refer to as gates
\begin{align*}
	\beta_d=\theta_dz_d,\quad
	\theta_d\neq 0,\quad
	z_d\in [0,1].
\end{align*}
As $\|\bm\beta\|_0=\|\vz\|_0$, one often considers the following regularized least square formulation of \eqref{eq:l0-problem} for estimating the parameters(here $\EE_{\vz}$ denotes expectation with respect to distribution of $\vz$)
\begin{align}\label{eq:gated-objective}
	\EE_{\vz}\qth{\frac{\|\mX(\bm\theta\odot\vz)-\vy\|_2^2}N +\lambda_N\|\vz\|_0},
\end{align}
where $\odot$ denotes the coordinatewise product
$$(a_1,\dots,a_D)\odot (b_1,\dots,b_D)=(a_1b_1,\dots,a_Db_D).$$
This $\odot$ product is also known as the Hadamard product in the literature; see \cite{Hadamard} for historical uses. The regularizor $\lambda_N=\lambda_N(K)$ depends on $N,K$ and is usually chosen based on cross-validation prior to learning the parameters. The optimization of the above risk function is done over $\bm\theta$ and the parameters controlling the distribution of $\vz$. Even though the Bernoulli distribution seems a natural choice for modeling $\vz$, minimization over discrete distributions is usually difficult and it is natural to consider continuous relaxation of the Bernoulli distribution. Examples of such relaxations include Concrete distribution \cite{jang2016categorical,maddison2016concrete}, Hard-Concrete distribution \cite{louizos2017learning} etc. In recent work, \cite{yamada20a} discussed that the previous two distributions induce high variance in the corresponding Bernoulli approximations and as an alternative proposed the STG, a differentiable relaxation of the Bernoulli gates, given by $\vz=\vz(\bs\mu) = \sth{\ervz_d(\mu_d)}_{d=1}^D$ with
\begin{equation}
	\ervz_d(\mu_d) = \max(0, \min(1, \mu_d + {\delta}_d)),
	\quad {\delta}_d \sim \mathcal{N}(0, \tau^2),
\end{equation}
where $\tau>0$ is fixed throughout training. In view of this specification, the objective in \eqref{eq:gated-objective} can be written as
\begin{align}
	\mathsf{Risk}(\bs\theta,\bs\mu;\lambda_N,\tau)
	&=\frac{\EE_{\vz}\qth{\|\vy-\mX(\bs\theta\odot \vz(\bs \mu))\|_2^2}}N
	+\lambda_N\sum_{d=1}^{D}\Phi\pth{\frac{\mu_d}{\tau}} \label{eq:risk}
\end{align}
where $\Phi$ is the standard normal cumulative distribution function. We minimize the above Risk with respect to $\bs\theta,\bs\mu$ to get $\hat{\bs\theta}, \hat{\bs \mu}$. The estimate of $\bs\beta^*$ is then provided by $\hat{\bs\theta}\odot \vz(\hat{\bs \mu})$.

\section{The Projected-STG algorithm}
\label{sec:algorithm}

Turning to computational aspects we note that the usual algorithm for obtaining minimizer of $\Risk(\bm\theta,\bm\mu;\lambda_N,\tau)$ \cite[Algorithm 1]{yamada20a} updates $(\bm\theta,\bm\mu)$ via simultaneous gradient descent on the parameters until convergence. The algorithm proposed by \cite{yamada20a} is designed for finding a subset of informative features under a general non-linear model. For the special case of linear models, it turns out that at each stage of the parameter update scheme, instead of changing $\bs\theta$ with a gradient direction, we can directly find the optimum choice via solving a simple quadratic objective. When the design matrix has full rank the optimal choice has a closed form expression that is similar to the classical projection-based estimator of $\beta^*$ given by $\hat{\bs\beta}=(\bm\mX^\sfT\bm\mX)^{-1}\bm\mX^\sfT \vy$.

\begin{theorem}\label{thm:optim-theta}
	Given fixed $\bm\mu,\bm\mX$, the minimizer $\hat {\bsym \theta}$ of \eqref{eq:risk} also minimizes
	$${\bm\theta}^{{\sfT}}\pth{\mX^{{\sfT}}\mX\odot \EE_\vz\qth{\vz(\bm\mu)\vz(\bm\mu)^\sfT}}{\bm\theta}
	-2{\bm\theta}^{{\sfT}}
	\pth{(\mX^{{\sfT}}\vy)\odot \EE_\vz[\vz(\bm\mu)]}.$$ 
	If $\mX$ has full column rank, then, $\mX^{{\sfT}}\mX\odot \EE_\vz\qth{\vz(\bm\mu)\vz(\bm\mu)^\sfT}$ is invertible and the above quadratic has a unique minimizer
	$$\hat{\bm\theta}=\pth{\mX^{{\sfT}}\mX\odot \EE_\vz\qth{\vz(\bm\mu)\vz(\bm\mu)^\sfT}}^{-1}\pth{(\mX^{{\sfT}}\vy)\odot \EE_\vz[\vz(\bm\mu)]}.$$
\end{theorem}
Our new algorithm uses this improved update method. 
\begin{algorithm}[H]
	\caption{Projected-STG}\label{algo:proj-STG}
	{\bf Input}: $\mX\in\reals^{N\times D},\vy\in \reals^N,K,\lambda_N$, number of epochs $R$, Monte Carlo samples $M,L$, variance of the gates $\tau$, learning rate $\gamma$. Initialize $\mu_d=0.5$ for $d=1,\dots,D$\\
	{\bf Output}: Trained parameters ${\bm\theta},{\bm\mu}$ and estimated support-set
	
	\begin{algorithmic}[1]
		\STATE Compute $\mQ=\EE_\vz\qth{\vz(\bm\mu)\vz(\bm\mu)^\sfT},
		\vq=\EE_\vz[\vz(\bm\mu)]$ using $M$ samples.    
		\STATE Update $\bm\theta \coloneqq \qth{{\mX}^\sfT {\mX} \odot \mQ }^{-1} 
		\qth{({\mX}^\sfT \vy) \odot \vq}$
		\label{line:theta-choice}    
		\FOR{$\ell=1,\dots,L$}
		\FOR{$d=1,\dots,D$}
		\STATE Sample $\delta_d^{(\ell)}\sim \calN(0,\tau^2)$
		\STATE Compute $z_d^{(\ell)}=\max\pth{0,\min\pth{1,\mu_d+\delta_d^{(\ell)}}}$
		\ENDFOR
		\ENDFOR
		\STATE Compute $V=\frac 1{NL}\sum_{\ell=1}^{L}\|\vy-\mX(\bs\theta\odot\vz^{(\ell)})\|_2^2+\lambda_N\sum_{d=1}^{D}\Phi\pth{\frac{\mu_d}{\tau}}$
		\STATE Update $\bm\mu\coloneqq\bm\mu-\gamma\nabla_{\bm\mu}V$
		\STATE Repeat $R$ epochs
		\STATE $\hat{\bs\beta}={\bs\theta}\odot \vz(\bs\mu)$. Estimate support-set by coordinates of $\hat{\bs \beta}$ with $K$ highest entries.
	\end{algorithmic}
\end{algorithm}
\begin{remark}
	The matrix ${\mX}^\sfT {\mX} \odot \mQ$ in the above algorithm is invertible when $X$ has full column rank. Otherwise, we update $\bs\theta$ by choosing the optimal candidate from the set of minimizers of the quadratic function in \prettyref{thm:optim-theta}.
\end{remark}

\begin{proof}[Proof of \prettyref{thm:optim-theta}]
	Fix $\bs\mu,\mX$ and denote $\vz(\bm\mu)$ by $\vz$ for simplicity. Then the optimal ${\bm\theta}$ for \eqref{eq:risk} is given by minimizer of $\EE_{\vz}\|\mX\pth{{\bm\theta}\odot {\vz}}-\vy\|_2^2$ and hence equivalently the minimizer of 
	$\EE_\vz\qth{({\bm\theta}^{{\sfT}}\odot {\vz}^{{\sfT}})(\mX^{{\sfT}}\mX)({\bm\theta}\odot {\vz})}
	-2\pth{{\bm\theta}^\sfT\odot\EE_\vz[{\vz}^\sfT]}(\mX^\sfT \vy)$. We note the followings. 
	\begin{itemize}
		\item For any two matrices $\mA,\mB $ we have $\mathsf{Trace}[\mA \mB ]=\mathsf{Trace}[\mB \mA]$ whenever the dimensions agree. In addition if $\mC$ is a symmetric matrix then for any $i$ we get  $\qth{(\mA\odot \mC)\mB }_{ii}
		=\sum_j\mA_{ij}\mC_{ij}\mB _{ji}
		=\sum_j\mA_{ij}\mC_{ji}\mB _{ji}
		=\qth{\mA(\mC\odot \mB )}_{ii}.$ In view this using linearity of expectation we get
		\begin{align*}
			\EE_\vz\qth{({\bm\theta}^{{\sfT}}\odot {\vz}^{{\sfT}})(\mX^{{\sfT}}\mX)({\bm\theta}\odot {\vz})}
			&=\EE_\vz\qth{\mathsf{Trace}\qth{({\bm\theta}^{{\sfT}}\odot \vz^{\sfT})(\mX^{\sfT}\mX)({\bm\theta}\odot \vz)}}
			\nonumber\\
			&=\EE_\vz\qth{\mathsf{Trace}\qth{({\bm\theta}\odot \vz)({\bm\theta}^{{\sfT}}\odot \vz^{{\sfT}})(\mX^{{\sfT}}\mX)}}
			\nonumber\\
			&=\EE_\vz\qth{\mathsf{Trace}\qth{({\bm\theta}{\bm\theta}^{\sfT}\odot \vz\vz^\sfT)(\mX^{{\sfT}}\mX)}}
			\nonumber\\
			&=\mathsf{Trace}\qth{({\bm\theta}{\bm\theta}^{\sfT}\odot \EE_\vz\qth{\vz\vz^\sfT})(\mX^{{\sfT}}\mX)}
			\nonumber\\
			&=\mathsf{Trace}\qth{({\bm\theta}{\bm\theta}^{\sfT})\pth{\mX^{{\sfT}}\mX\odot \EE_\vz\qth{\vz\vz^\sfT}}}
			\nonumber\\
			&={\bm\theta}^{\sfT}\pth{\mX^{{\sfT}}\mX\odot \EE_\vz\qth{\vz\vz^\sfT}}{\bm\theta}.
		\end{align*}
		\item For any three vectors $\va,\vb,\vc$ of same length we have $(\va^\sfT\odot\vb^\sfT)\vc=\sum_ia_ib_ic_i=\va^\sfT(\vb\odot\vc)$, which means 
		$$
		\pth{{\bm\theta}^\sfT\odot\EE_\vz[{\vz}^\sfT]}(\mX^\sfT \vy)
		={\bm\theta}^\sfT\pth{(\mX^\sfT \vy)\odot\EE_\vz[{\vz}]}.
		$$
	\end{itemize}
	Combining the above we conclude that our problem reduces to minimizing the objective
	\begin{align}
		{\bm\theta}^{{\sfT}}\pth{\mX^{{\sfT}}\mX\odot \EE_\vz\qth{\vz\vz^\sfT}}{\bm\theta}
		-2{\bm\theta}^{{\sfT}}\pth{(\mX^{{\sfT}}\vy)\odot \EE_\vz[\vz]}
		\label{eq:quad}
	\end{align}
	as needed to be shown.
	Now assume that $\mX$ has full column rank. Then it follows that $\bm\mX^\sfT\bm\mX$ is invertible. Note that $\EE_\vz\qth{\vz\vz^\sfT}$ is also positive definite as given any $\vu\neq 0$ we have
	\begin{align*}
		\vu^\sfT\EE_\vz\qth{\vz\vz^\sfT}\vu
		=\sum_{d=1}^Du_d^2\Var(z_d)+(\vu^\sfT\EE_\vz[\vz])^2>0,
	\end{align*}
	as $Var(z_d)>0$ for $d=1,\dots,D$.
	Then using Schur's Theorem \cite[Theorem 3.1]{S73} we get that $\mX^{{\sfT}}\mX\odot \EE_\vz\qth{\vz\vz^\sfT}$ is positive definite. For any positive definite matrix $\mA$ and vector $\vb$ we have
	\begin{align*}
		{\bm\theta}^{{\sfT}}\mA{\bm\theta}
		-2{\bm\theta}^{{\sfT}}\vb
		=\|\mA^{\frac 12}\bs\theta-\mA^{-\frac 12}\vb\|_2^2
		-\|\mA^{-\frac 12}\vb\|_2^2,
	\end{align*}
	which implies ${\bm\theta}^{{\sfT}}\mA{\bm\theta}
	-2{\bm\theta}^{{\sfT}}\vb$ has unique minimizer $\hat{\bs\theta}=\mA^{-1}\vb$. Then it follows that \eqref{eq:quad} has the unique minimizer given by
	$
	\hat {\bs\theta}=(\mX^{{\sfT}}\mX\odot \EE_\vz\qth{\vz\vz^\sfT})^{-1}\pth{(\mX^{{\sfT}}\vy)\odot \EE_\vz[\vz]}.
	$
\end{proof}

	
	
	
	

\section{Theoretical guarantees of the STG}

The results in this paper also apply to high-dimensional settings where $D=D(N), K=K(N)$ are allowed to grow with $N$. Let $\hat {\bs\theta},\hat {\bs\mu}$ be a minimizer of \eqref{eq:risk}. We provide guarantees for our final estimator $\hat{\bm\beta}=\hat{\bm\theta}\odot\vz(\hat{\bm\mu})$. Our analysis follows closely along the lines of analysis for the SCAD penalty presented in \cite{huang2007asymptotic}. 

For any symmetric matrix $\mA$, define its spectral norm
$$
\spnorm{\mA}\eqdef\max\{|\rho|:\rho \text{ is an eigenvalue of } \mA \}.
$$
Given any positive-definite matrix $\mA$, let $\rho_{\min}>0$ denote the smallest eigenvalues of $\mA$. The following result is the central idea for deducting convergence guarantees of the STG.
\label{sec:theory}
\begin{theorem}[Convergence to the ground truth]\label{thm:prob-convergence}
	Suppose that $D/N\to 0,K\lambda_n\to 0$ and the design matrix $\mX$ satisfies $$\spnorm{\frac 1N \mX^{\sfT}\mX-\EE\qth{\frac 1N \mX^{\sfT}\mX}} \xrightarrow[]{P} 0,\quad \rho=\lim_{N\to\infty}\rho_{\min}(\EE\qth{\frac 1N \mX^{\sfT}\mX})>0.$$ Then $\hat{\bm \beta}$ converges to $\bm\beta^*$ in probability.
\end{theorem}
\begin{remark}
	The condition regarding the convergence of $\frac 1N \mX^T\mX$ is satisfied in many general cases. For example, consider the cases when the entries of the design matrix are fixed or are independently generated with bounded second moments. Condition on the regularize similar to $\lim_{N\to\infty }K\lambda_N=0$ is standard in the literature; see \cite{huang2007asymptotic} for instance. For our simulation studies we choose $\lambda_N=C\sqrt{\log(D-K)\log(K)\over N}$ for some suitable constant $C$ chosen via cross-validation, which implies that the theoretical results will hold whenever $\lim_{N\to\infty}K^{1.5}\sqrt{\log(D-K)\over N}\to 0$. Our simulations show that the STG estimators perform well even when the dimension $D$ is much bigger than the sample size. However, the high-dimensional case analysis is beyond this paper's scope.
\end{remark}
\begin{proof}
	Note that for any symmetric matrix $\mA$ we have $\spnorm{\mA}=\max_{\|\vy\|_2=1}\vy^\sfT \mA\vy$ and for any non-negative definite matrix $\mA$ we have $\rho_{\min}(\mA)=\min_{\|\vy\|_2=1}\vy^\sfT \mA\vy$. These properties imply that given any two non-negative definite matrices $\mA,\mB$ we have
	\begin{align} 
		\label{eq:norm_prop}
		\rho_{\min}(\mA) &= \min_{\|\vy\|_2=1}\vy^\sfT \mA\vy
		\nonumber\\
		&=\min_{\|y\|_2=1}\pth{\vy^\sfT \mB\vy + \vy^\sfT (\mA-\mB)\vy}
		\nonumber\\
		&\leq \min_{\|y\|_2=1}\pth{\vy^\sfT \mB\vy + \spnorm{\mA-\mB}}
		= \rho_{\min}(B)+\spnorm{A-B}.
	\end{align}
	Define $\sS_{N}=\{\mX:\spnorm{\frac 1N \mX^{\sfT}\mX-\EE\qth{\frac 1N \mX^{\sfT}\mX}}\leq \rho/2\}$. Hence for any $\mX\in\sS_N$ using \eqref{eq:norm_prop} with $\mA=\EE\qth{\frac 1N\mX^\sfT\mX},\mB=\frac 1N\mX^\sfT\mX$ we have $$\rho_{\min}(\frac 1N\mX^{\sfT}\mX)\geq \rho_{\min}\pth{\EE\qth{\frac 1N \mX^{\sfT}\mX}}-\spnorm{\frac 1N \mX^{\sfT}\mX-\EE\qth{\frac 1N \mX^{\sfT}\mX}}\geq \rho/2.$$ 
	Let us assume for simplicity that the first $K$ coordinates of $\bm\beta^*$ contain its support and $\sigma=\tau=1$. Note that any other such choice can be analyzed similarly. 
	Choose $\bs \mu^*=(N1_K,-N1_{{D}-K})$ where $1_{a}$ is an $``a"$ length vector of all 1's and $N$ is some arbitrarily large number to be chosen later. 
	From Mill's ratio bound \cite{G41} on $\Phi(-N)=1-\Phi(N)$ : $\Phi(-N)
	\leq \frac 1{\sqrt {2\pi}N} e^{-{N^2\over 2}}$
	$$\PP\qth{\vz(\bs \mu^*)\neq (1_{K},0_{D-K})}
	\leq D\Phi(-N)
	\leq \frac D{\sqrt {2\pi}N} e^{-{N^2\over 2}}.$$
	Letting $N\to\infty$, the last display implies 
	\begin{align}
		&\PP\qth{\sth{\mX\notin \sS_N}\cup\sth{\vz(\bs \mu^*)\neq (1_{K},0_{D-K})}}
		\leq 
		\PP\qth{\sS_N}+\PP\qth{\vz(\bs \mu^*)\neq (1_{K},0_{D-K})}
		\to 0.
		\label{eq:m0}
	\end{align}
	We first restrict the design matrix $\mX$ on the set $\sS_N$ and fix $\vz(\bs\mu^*)=(1_{K},0_{D-K})$. Note that for every $\mX\in\sS_N$ the matrix ${\mX^{{\sfT}}\mX}$ is invertible as its smallest eigenvalue is positive. Hence using optimality of $\hat{\bs\theta},\hat{\bs\mu}$ and the fact $\bm\beta^*\odot {\vz}(\bs  \mu^*)=\bm\beta^*$ we get
	for each fixed realization of $\vy,\bs\epsilon$
	\begin{align*}
		0&\geq N\sth{\Risk(\bs{\hat \theta},\hat{\bs \mu};\lambda_N,1)-\Risk(\bm\beta^*,\bs \mu^*;\lambda_N,1)}
		\nonumber\\
		&\geq {\EE_{\vz}\qth{\|\vy-\mX{\hat{\bm \beta}}\|_2^2}
			-{\|\vy-\mX\bm\beta^*\|_2^2}}
		+ N\lambda_N\pth{\sum_{d=1}^D\sth{\Phi\pth{\hat {\bs\mu}  }-\Phi\pth{\bs \mu^*}}}
		\nonumber\\
		&\geq \EE_{\vz}\qth{\|\mX({\hat{\bm \beta}}-\bm\beta^*)\|_2^2 -2\bs\epsilon^{{\sfT}}\mX\pth{\hat{\bm \beta}-\bm\beta^*}}
		-N\lambda_N\pth{K+(D-K)\Phi\pth{-N}}
		\nonumber\\
		&\geq \EE_{\vz}\qth{\|(\mX^{{\sfT}}\mX)^{\frac 12}({\hat{\bm \beta}}-\bm\beta^*)\|_2^2 -2\bs\epsilon^{{\sfT}}\mX\pth{{\hat{\bm \beta}}-\bm\beta^*}}
		-NK\lambda_N-{\lambda_NDe^{-{N^2\over 2}}\over \sqrt{2\pi}}
		\nonumber\\
		&\geq \EE_{\vz}\qth{\|(\mX^{{\sfT}}\mX)^{\frac 12}({\hat{\bm \beta}}-\bm\beta^*)-(\mX^{{\sfT}}\mX)^{-\frac 12}\mX^\sfT\bs\epsilon\|_2^2} 
		-\bs\epsilon^{{\sfT}}\mX(\mX^{{\sfT}}\mX)^{-1}\mX^{{\sfT}}\bs\epsilon
		-NK\lambda_N
		-{\lambda_NDe^{-{N^2\over 2}}\over \sqrt{2\pi}},
	\end{align*}
	and hence using $\|\va+\vb\|_2^2\leq 2(\|\va\|_2^2+\|\vb\|_2^2)$ we get
	\begin{align}
		&\EE_{\vz}\|(\mX^{{\sfT}}\mX)^{\frac 12}(\bm\beta^*-{\hat{\bm\beta}})\|_2^2
		\leq 4\bs\epsilon^{{\sfT}}\mX(\mX^{{\sfT}}\mX)^{-1}\mX^{{\sfT}}\bs\epsilon+NK\lambda_N
		+{\lambda_NDe^{-{N^2\over 2}}\over \sqrt{2\pi}}.
		\label{eq:m1}
	\end{align}
	Using the fact for any symmetric matrix $\mA$ and vector $y\neq 0$
	\begin{align*}
		\|\mA\vy\|^2\geq \rho_{\min}(\mA^2) \|y\|_2^2
	\end{align*}
	we get
	\begin{align}
		\|(\mX^{{\sfT}}\mX)^{\frac 12}(\bm\beta^*-{\hat{\bm \beta}})\|_2^2 
		\geq N\rho_{\min}(\frac 1N \mX^T\mX)\|\bm \beta^*-\hat{\bm \beta}\|_2^2
		\geq N\frac {\rho}2\|\bm \beta^*-\hat{\bm \beta}\|_2^2
		\label{eq:m2}
	\end{align}
	Using linearity of expectation, $\EE_{\bs\epsilon}[\bs\epsilon\bs\epsilon^\sfT]=\mI_n$, and commutativity of $\mathsf{Trace}$ operator we get for each $\mX\in\sS_N$
	\begin{align*}
		\EE_{\bs\epsilon}\qth{\bs\epsilon^{\sfT}\mX(\mX^{\sfT}\mX)^{-1}\mX^{\sfT}\bs\epsilon} 
		&=\EE_{\bs\epsilon}\qth{\mathsf{Trace}\pth{\bs\epsilon^{\sfT}\mX(\mX^{\sfT}\mX)^{-1}\mX^{\sfT}\bs\epsilon}}
		\nonumber\\
		&=\EE_{\bs\epsilon}\qth{\mathsf{Trace}\pth{\mX(\mX^{\sfT}\mX)^{-1}\mX^{\sfT}\bs\epsilon\bs\epsilon^{\sfT}}}
		\nonumber\\
		&=\mathsf{Trace}\pth{\mX(\mX^{\sfT}\mX)^{-1}\mX^{\sfT}\EE\qth{\bs\epsilon\bs\epsilon^{\sfT}}}
		=\mathsf{Trace}\pth{\mX(\mX^{\sfT}\mX)^{-1}\mX^{\sfT}} 
		=D.
	\end{align*}
	In view of \eqref{eq:m1} and \eqref{eq:m2} this implies
	\begin{align*}
		&\EE\qth{\left.\|{\bm\beta^*}-\hat{\bm\beta}\|_2^2\right|\sth{\mX\in \sS_N}\cap\sth{\vz(\bs \mu^*)= (1_{K},0_{D-K})}}
		\nonumber\\
		&\leq 
		\frac 2{N\rho}\pth{\EE\qth{4\bs{\epsilon}^\sfT \mX(\mX^\sfT \mX)^{-1}\mX^\sfT\bs{\epsilon}|\mX\in\sS_N}
			+NK\lambda_N+\frac {\lambda_N De^{-\frac{N^2}2}}{\sqrt{2\pi}}}
		\nonumber\\
		& \leq \frac {2}\rho\pth{\frac{4D}N+K\lambda_N+{\lambda_NDe^{-{N^2\over 2}}\over \sqrt{2\pi}N}}
		\to 0.
	\end{align*}
	In view of the Markov inequality, this will imply for every $c>0$
	\begin{align*}
		&\PP\qth{\|{\bm\beta^*}-\hat{\bm\beta}\|_2^2>c}
		\nonumber\\
		&\leq  \PP\qth{\left.\|{\bm\beta^*}-\hat{\bm\beta}\|_2^2>c\right|\sth{\mX\in \sS_N}\cap\sth{\vz(\bs \mu^*)= (1_{K},0_{D-K})}}
		+\PP[\sth{\mX\notin \sS_N}\cup\sth{\vz(\bs \mu^*)\neq (1_{K},0_{D-K})}]
		\nonumber\\
		&\leq \frac{\EE\qth{\left.\|{\bm\beta^*}-\hat{\bm\beta}\|_2^2\right|\sth{\mX\in \sS_N}\cap\sth{\vz(\bs \mu^*)\neq (1_{K},0_{D-K})}}}c
		+\PP[\sth{\mX\notin \sS_N}\cup\sth{\vz(\bs \mu^*)\neq (1_{K},0_{D-K})}]
		\nonumber\\
		&\to 0,
	\end{align*}
	as $N$ tends to infinity. 
\end{proof}

\begin{theorem}[Support recovery]\label{thm:support-recovery}
	Suppose that $\bs\beta^*$ has $K$ non-zero entries and let 
	$\max_{i:\beta^*_i\neq 0}\abs{\beta_i^*}>\eta>0$ for some absolute constant $\eta$. Then under the conditions in \prettyref{thm:prob-convergence}, the indices of $\hat{\bm \beta}$ with $K$-largest magnitudes (tie broken arbitrarily) recover the exact support of $\bm\beta^*$ with probability tending to 1.
\end{theorem}

\begin{proof}
	Suppose that the locations of $K$-largest $\hat\beta_i$'s (magnitude wise) do not match the support of $\beta^*$. Then there exist two indices $i,i'$ in $\{1,\dots,D\}$ such that $|\hat\beta_i|\geq |\hat\beta_{i'}|$ (incorporating scenario of ties) and $\beta^*_i=0,\beta^*_{i'}\neq 0$ (and hence $|\beta^*_{i'}|>\eta$). Then it follows that
	\begin{itemize}
		\item if $|\hat\beta_i|\geq {\eta}/2$ we have $|\hat\beta_i-\beta^*_i|>{\eta}/2$
		\item if $|\hat\beta_i| < {\eta}/2$ then ${\eta}/2>|\hat\beta_{i'}|$ and so 
		$|\hat\beta_{i'}-\beta^*_{i'}|>{\eta}/2$.
	\end{itemize}
	Combining these we get that unsuccessful recovery implies ${\eta}/2\leq \|\hat{\bm\beta}-\bm\beta^*\|_2$, which occurs with negligible probability by \prettyref{thm:prob-convergence}.
\end{proof}

The following lemma is a direct consequence of the above theorems.
\begin{corollary}
	Suppose that the design matrix $\mX$ is fixed with $\mX^\sfT\mX$ invertible. Then the STG-estimator recovers the support of $\bm\beta^*$ whenever $\frac DN\to 0$ and $K\lambda_N\to 0$. 
\end{corollary}

The results for random design-matrix structures follow from the above theorem and the results on the convergence of sample covariance matrices. In this setup, we present the results for two specific cases for the design matrix: (1) design matrices with sub-Gaussian entries, and (2) design matrices with the entries having moment constraints. Let us call a random vector $\vx$ to be Sub-Gaussian if for any vector $\vy$ in the $D$ dimensional unit sphere $\calS^{D-1}$ the random variable $\vx^\sfT y$ has Sub-Gaussian tail
\begin{align}
	\PP[|\vx^\sfT \vy|>t]\leq 2e^{-t^2\over L^2},\quad \forall \vy\in \calS^{D-1}
\end{align}
for some finite constant $L$. Similarly, let us define a random vector $\vx$ to have $q$-th moment bounded if for any vector $\vy$ in the $D$ dimensional unit sphere $\calS^{D-1}$ the random variable $\vx^\sfT y$ has bounded $q$-th moment
\begin{align}
	\|\vx\|_2\leq \alpha\sqrt D \text{ a.s.},\quad \EE[|\vx^\sfT \vy|^q]\leq L^q,\quad \forall \vy\in \calS^{D-1}
\end{align}
with $\alpha,L$ being finite constants. We have the following results.
\begin{corollary}\label{cor:SubG}
	Suppose that the columns of the design matrix $\mX$ are sub-gaussian with covariance matrix $\bm\Sigma$ and mean vector $\bm \mu$, with $\rho_{\min}(\bm\Sigma+\bm \mu\bm\mu^{\sfT})>0$. Then the STG-estimator recovers the support of $\bm\beta^*$ if $\frac DN\to 0$ and $K\lambda_N\to 0$.. 
\end{corollary}

\begin{corollary}\label{cor:moments}
	Suppose that the columns of the design matrix $\mX$ have finite $q$-th moment with $q>4$.  with covariance matrix $\bm\Sigma$ and mean vector $\bm \mu$, with $\rho_{\min}(\bm\Sigma+\bm \mu\bm\mu^{\sfT})>0$. Then the STG-estimator recovers the support of $\bm\beta^*$ if $(\log\log D)^2\pth{D\over N}^{\frac 12-\frac 2q}\to 0$ and $K\lambda_N\to 0$.. 
\end{corollary}

\begin{proof}[Proof of \prettyref{cor:SubG} and \prettyref{cor:moments}]
	The convergence in probability condition on $\frac 1N \mX^\sfT\mX$, as required by \prettyref{thm:prob-convergence}, is guaranteed by \cite[Proposition 2.1]{vershynin2012close} in the Sub-Gaussian case and is guaranteed by \cite[Theorem 1.2]{vershynin2012close} in the moment-constrained case. Hence the results follow.
\end{proof}

\section{Simulation studies}
\label{sec:simulation}
\subsection{Synthetic Data}
We evaluate the performance guarantees of our algorithm based on the probability of {\it exact support recovery}, i.e., when the algorithm correctly identifies all non-zero entries of $\bs\beta^*$. Once the support is identified, estimating coefficients becomes much less challenging as the number of unknowns becomes relatively smaller. This quantity is estimated over a batch of runs of the algorithm by computing the ratio of the number of exact recovery events to the number of total runs. We consider comparing our method Projected-STG (Proj-STG) against LASSO \cite{tibshirani1996regression}, OMP \cite{omp}, Randomized OMP (Rand-OMP) \cite{randomp}, SCAD \cite{scad}, and the original STG \cite{yamada20a}. The success rate of our method is compared against the algorithms mentioned above in two empirical studies -- (a) as the number of data points $N$ grows large, and (b) as the sparsity level $K$ grows large. One should expect that in (a), the target probabilities will grow monotonically to 1 as $N$ grows large, a direct consequence of \prettyref{thm:support-recovery} when non-zero coordinates of $\bs\beta^*$ are bounded away from 0. In (b), the success rate should decrease as $K$ takes larger values (keeping everything else fixed) which signifies difficulty recovering sporadic null coordinates when the signal is strong. As we will see below, our simulations identify these behaviors as well. We study the performance of our algorithm using the design matrices listed below.

{{\bf Gaussian design matrix:}} In both simulations, we fixed $D=64$ and varied either $N$ or $K$. For each run, we re-sampled the dataset in the following way. We generated the design matrix $\mX \in \mathbb{R}^{N \times D}$ by drawing each of $X_{ij}$'s independently from a standard Gaussian distribution. To generate the signal $\bs\beta^*$, we first assigned $0$ value to $D - K$ randomly chosen coordinates. Next, the nonzero values of $\bs\beta^*$ were drawn from a symmetric Bernoulli distribution with values $1$ or $-1$. Finally, the target vector $\vy$ is simulated using \eqref{eq:main-model} separately with two different values of $\sigma=0.5,1$. We run $100$ simulations of \prettyref{algo:proj-STG} with $\tau=0.5$, using 20 Monte Carlo samples to approximate the expectations in $\mQ,\vq$ in \prettyref{algo:proj-STG}. For updating values of $\bs\mu$, instead of fixed step size ($\gamma$ in \prettyref{algo:proj-STG}) we use the Adam optimizer \cite{kingma2014adam} which stochastically chooses an improved step size. For LASSO, the regularization parameter is set to $\lambda_{N,0} = \sqrt{\frac{2 \sigma^2 \log (D-K) \log (K)}{N}}$ as suggested in \cite[Section VII]{wainwright09}.
For Projected-STG, we use the regularization parameter $\lambda_N= C \lambda_{N,0}$, where $C$ is a constant selected using a cross validated grid search in $[0.1, 10]$. 

In the first simulation, we fixed $K=\lceil 0.40 D^{0.75} \rceil=10$ (as suggested in \cite[Section I(B)]{wainwright09}) and varied $N$ within the interval $[10,100]$. \prettyref{fig:nvar} presents the estimated probabilities of exact support recovery, along with corresponding $90$ percent confidence bands (computed via bootstrap), for Projected-STG, LASSO, OMP, Rand-OMP, and SCAD. Given any estimate $\hat{\bs\beta}$ we choose $K$ of its coordinates with the largest magnitudes and use it as our estimated support. As demonstrated in this figure, the Projected-STG requires fewer samples for perfect support recovery when compared with the baselines. Moreover, as suggested by these results, our method performs significantly better than the other methods when $\sigma$ is large (i.e., the signal-to-noise ratio is small).

In the second simulation, we investigate the impact of varying the sparsity $K$ on the success rate of the estimators. We retain the same experimental setup as the previous experiment and fix the number of observations $N=40$ such that the system is underdetermined (recall that $D=64$). We varied $K$ within $[1,25]$. \prettyref{fig:kvar} presents the probability of success along with $90$ percent confidence bands (computed via bootstrap). The figure suggests that the proposed model's outcomes can substantially improve when the sparsity level is high (large values of $K$).
\begin{figure}[htb!]
	\begin{center}
		\includegraphics[width=0.45\textwidth]{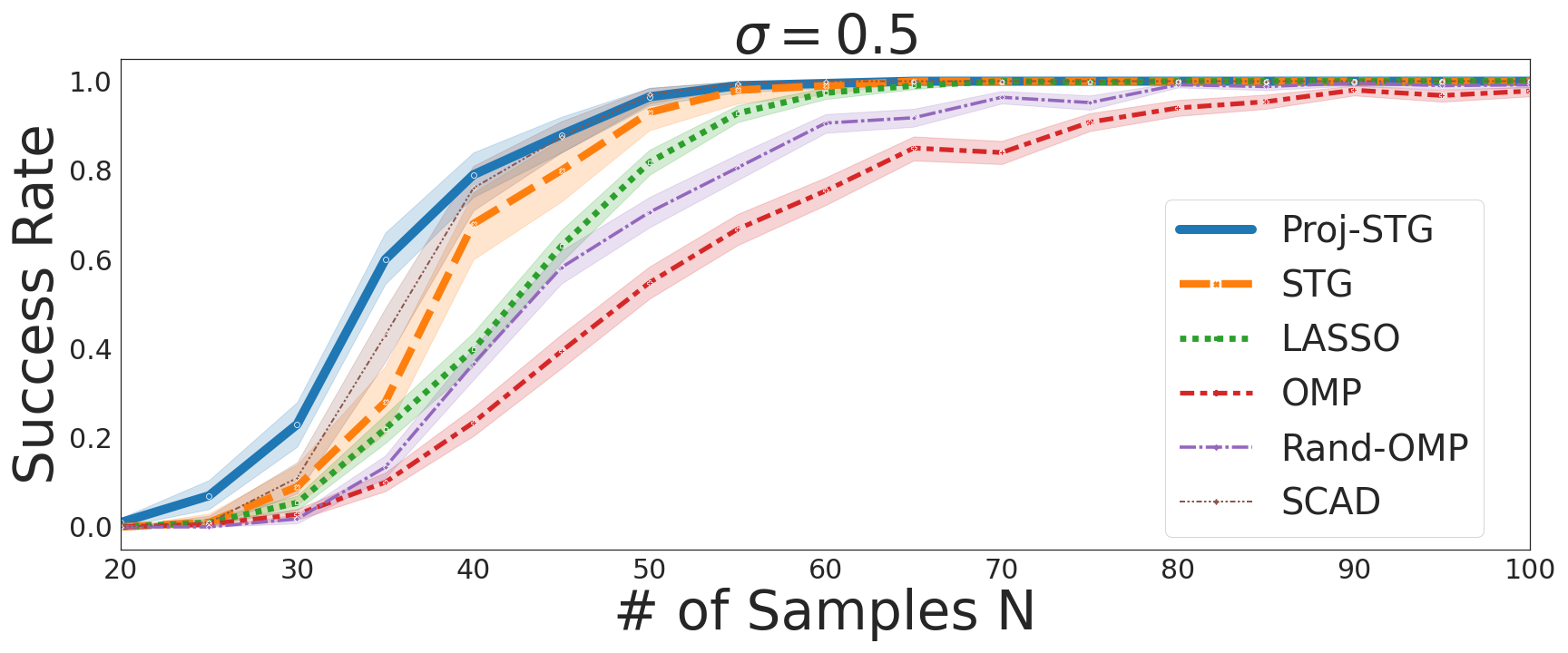}
		\includegraphics[width=0.45\textwidth]{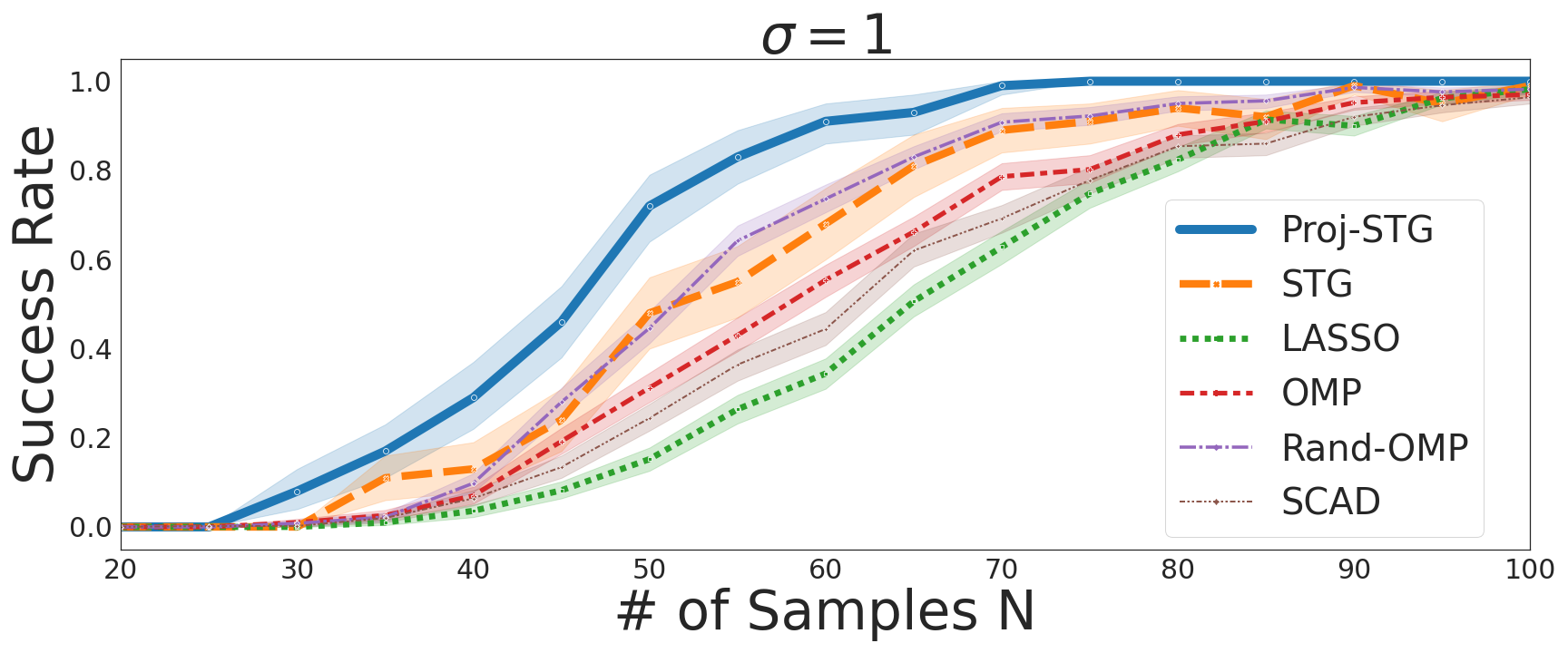}
		\caption{Probability of success in support recovery vs. number of samples $N$.}
		\label{fig:nvar}
	\end{center}
\end{figure} 

\begin{figure}[htb!]
	\begin{center}
		\includegraphics[width=0.45\textwidth]{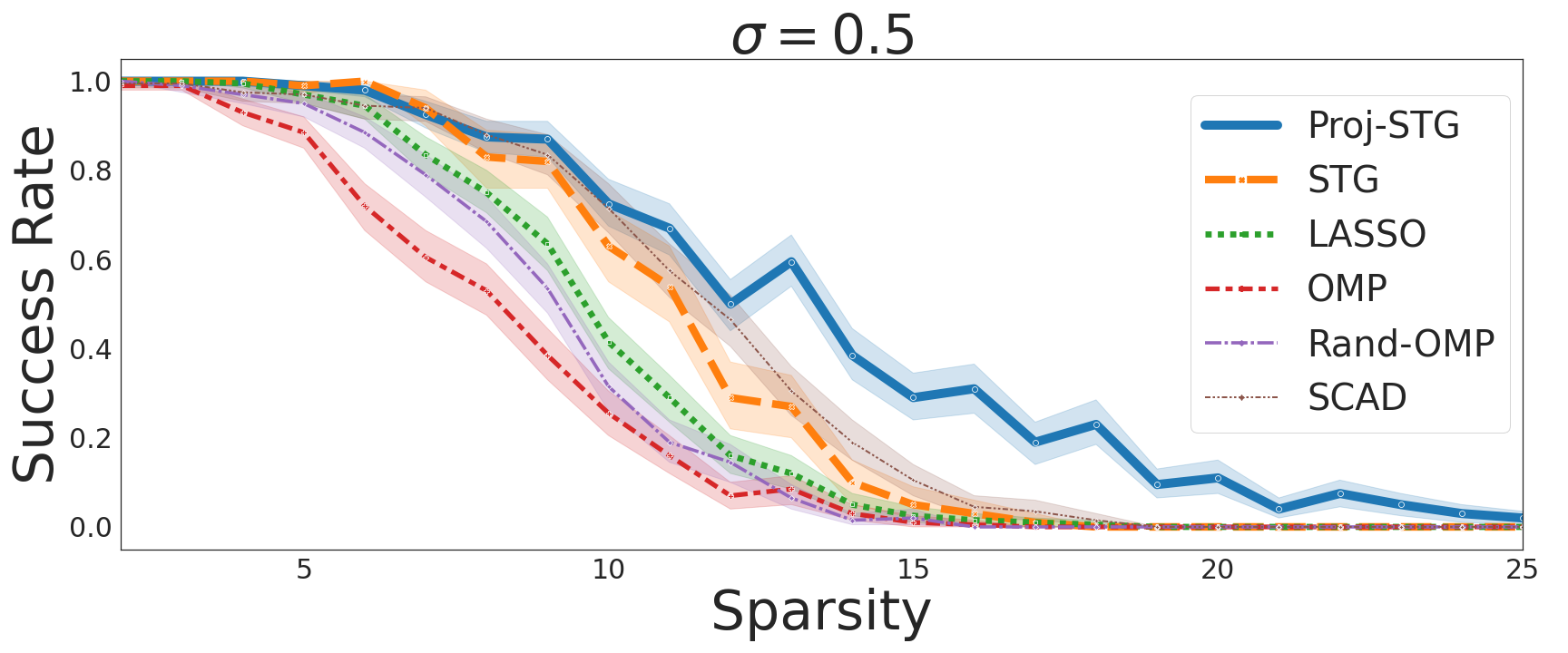}
		\includegraphics[width=0.45\textwidth]{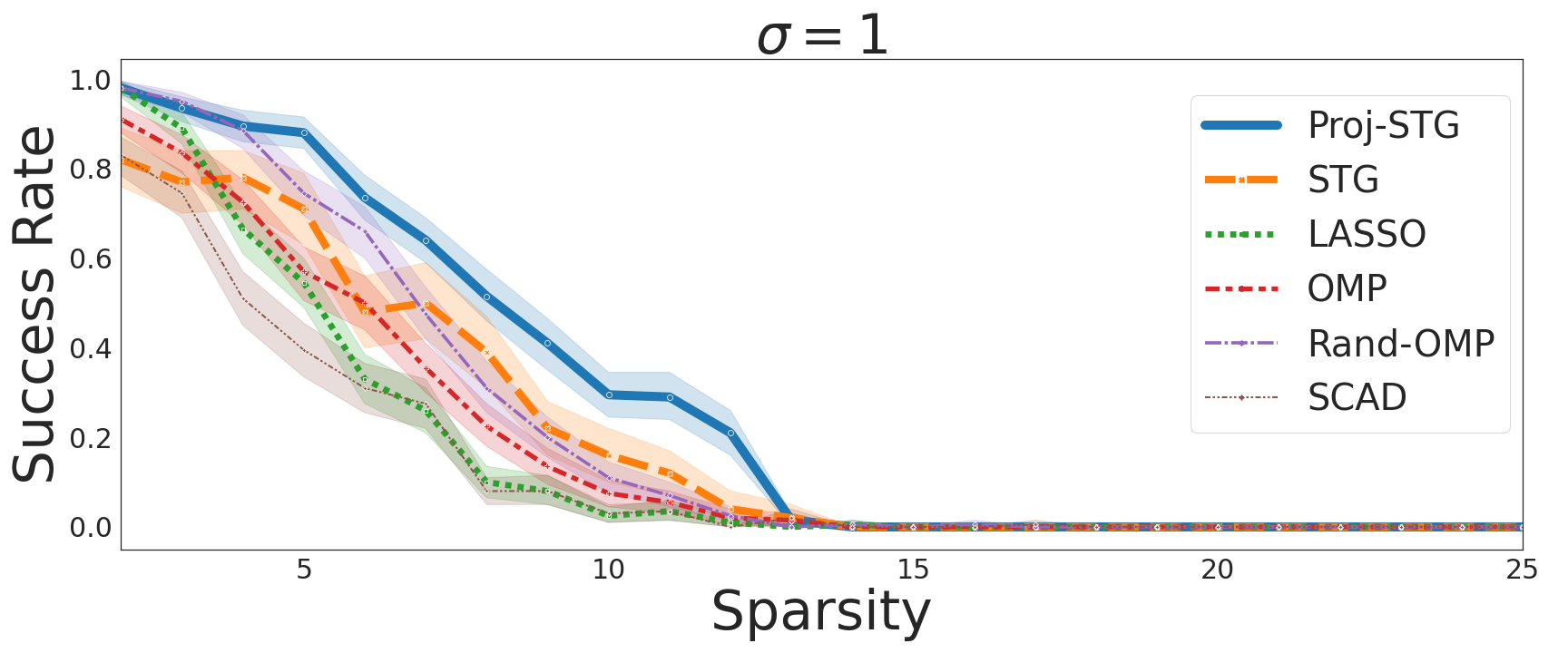}
		\caption{Probability of success in support recovery vs. sparsity level $K$.}
		\label{fig:kvar}
	\end{center}
\end{figure}

{{\bf Bernoulli design matrix}:} to demonstrate that the proposed approach can be applied to other design matrix structures, we further consider design matrices generated via the Bernoulli distribution. Specifically, we create the design matrix $\mX \in \mathbb{R}^{N \times D}$ by drawing each of $X_{ij}$'s independently from a fair Bernoulli distribution with values $\{-1,1\}$. Then, we repeat the first experiment and evaluate the effect of the number of samples on the probability of support recovery. \prettyref{fig:nvar-bin} presents the probability of success along with $90$ percent confidence bands for the Bernoulli design matrix.  

\begin{figure}[htb!]
	\begin{center}
		\includegraphics[width=0.45\textwidth]{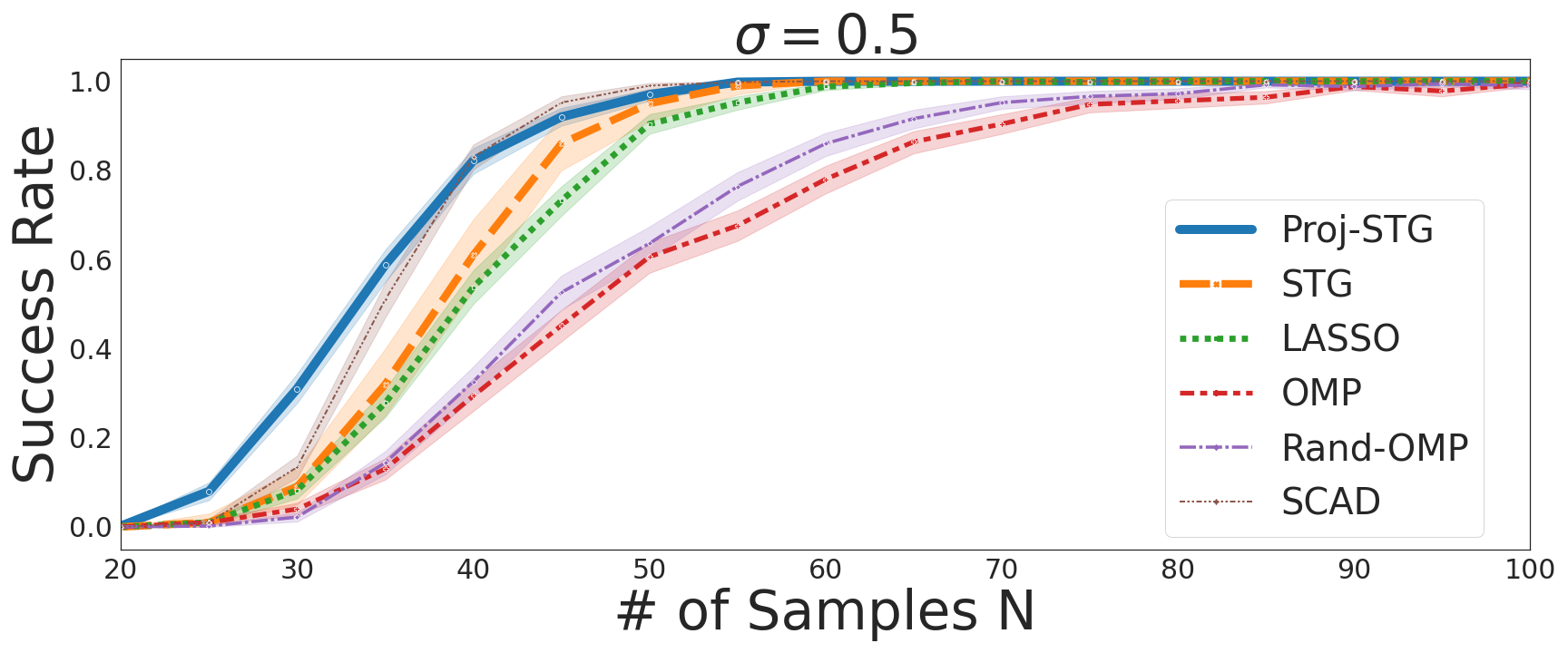}
		\includegraphics[width=0.45\textwidth]{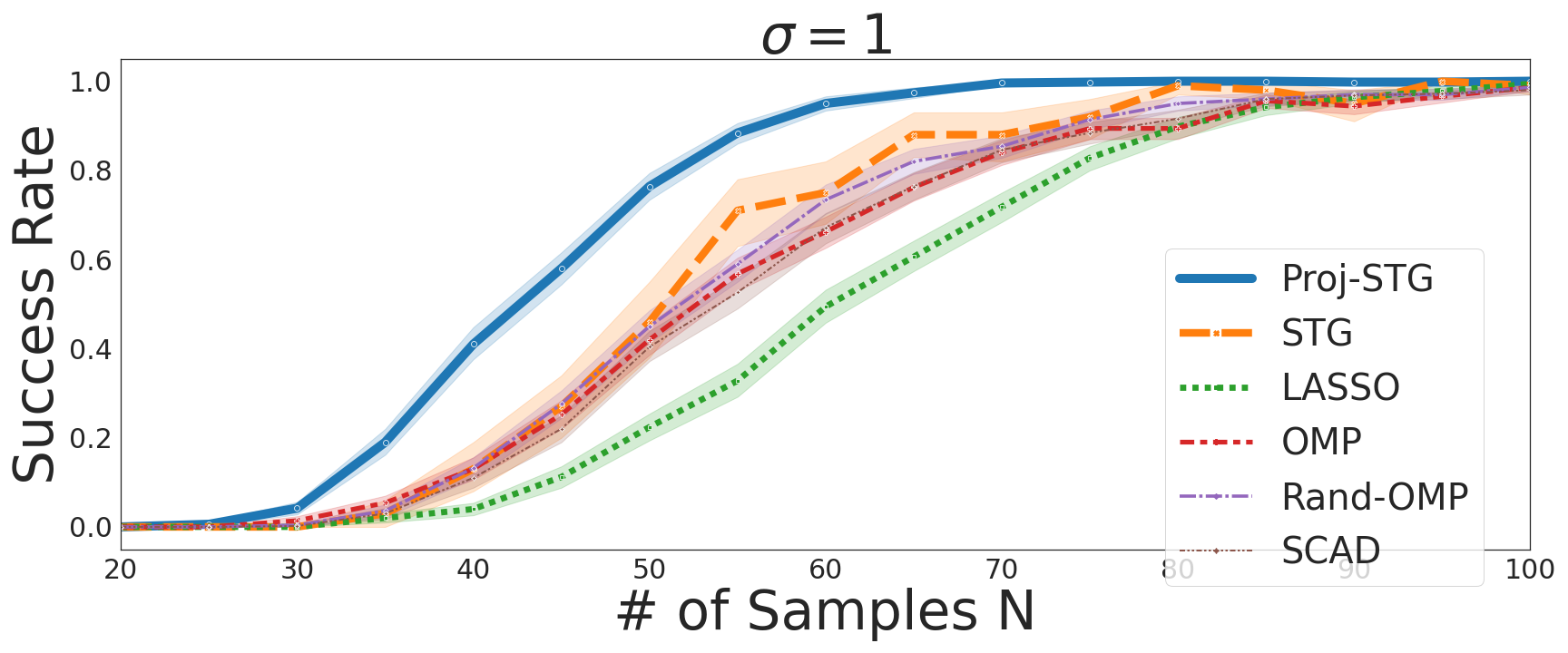}
		\caption{Probability of success in support recovery vs. number of samples.}
		\label{fig:nvar-bin}
	\end{center}
\end{figure}

\begin{figure}[htb!]
	\begin{center}
		\includegraphics[width=0.45\textwidth]{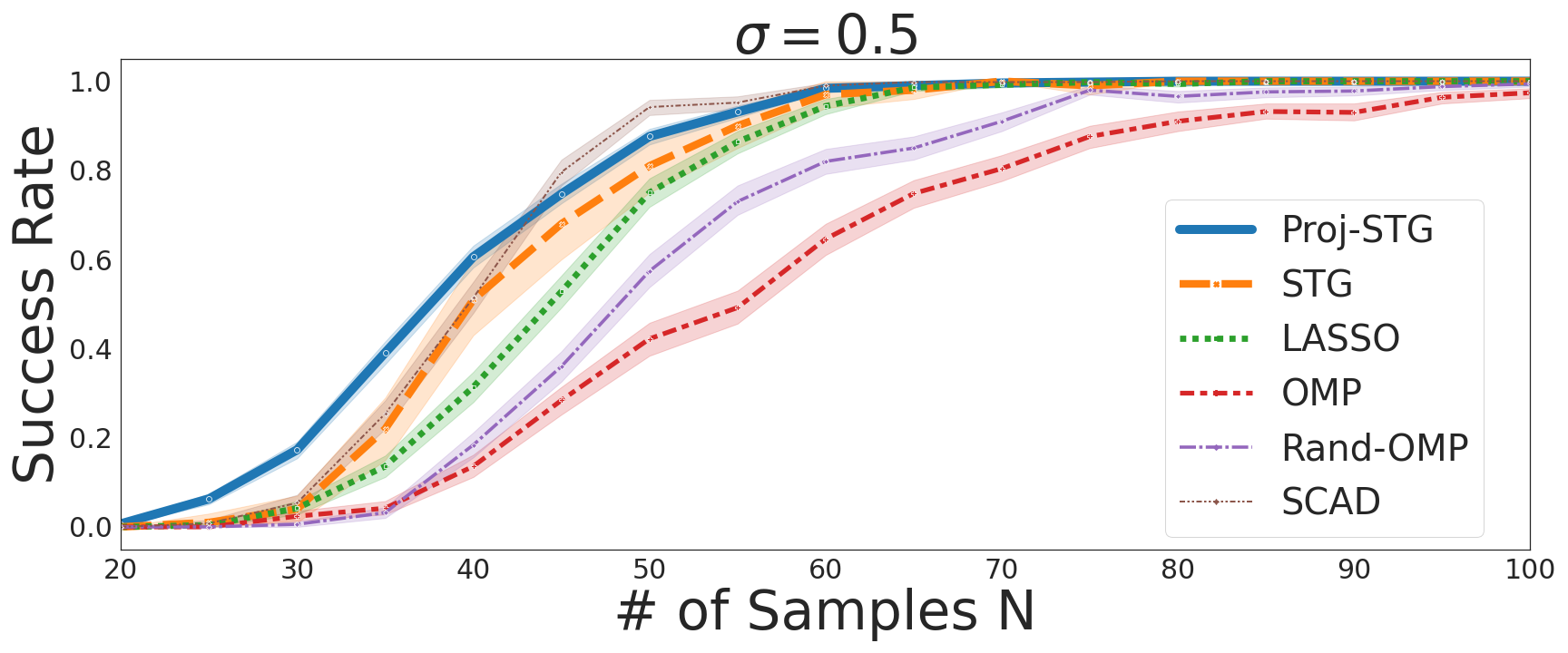}
		\includegraphics[width=0.45\textwidth]{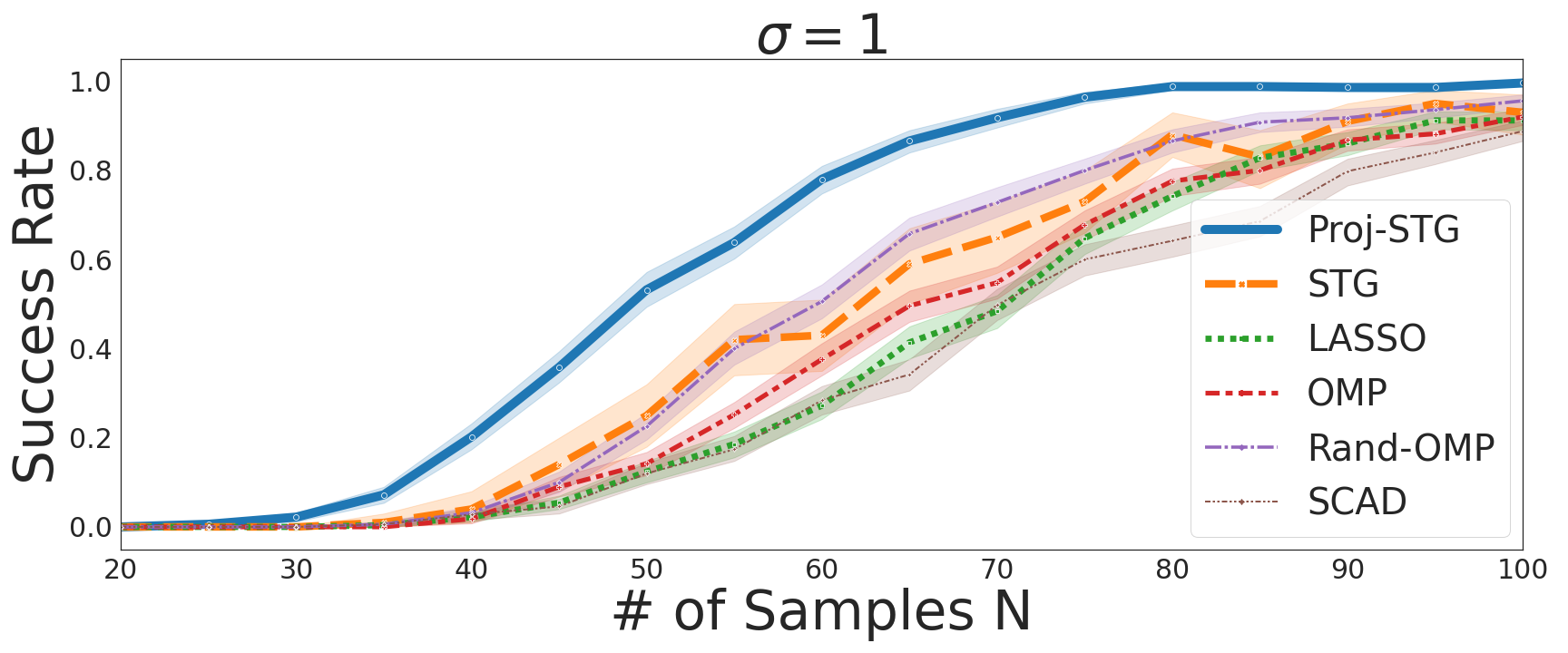}
		\caption{Probability of success in support recovery vs. number of samples.}
		\label{fig:nvar-cov}
	\end{center}
\end{figure}  

	

{{\bf Correlated design matrix:}}
In real-world data, the observed features are typically correlated. This typically makes the task of support recovery more challenging. We use a correlated design matrix to evaluate how the performance of the proposed method is affected by the correlation between features. Specifically, we construct a Toeplitz covariance matrix $\myvec{\Sigma}$, with values $\Sigma_{l,m}=\rho^{\|l-m\|}$, and $\rho=0.3$. Then, we draw the values of $X_{i,j}$ from $N(0,\myvec{\Sigma})$. We use the same settings as in the previous examples to generate $\myvec{\beta}^*$ and the additive noise, and we evaluate support recovery for different values of $N$. As indicated by \prettyref{fig:nvar-cov}, As expected, Proj-STG requires more samples for perfect recovery when the observed features are correlated. The proposed approach outperforms all baselines when the injected noise is strong.  

\begin{figure}
	\centering
	\includegraphics[width=0.45\textwidth]{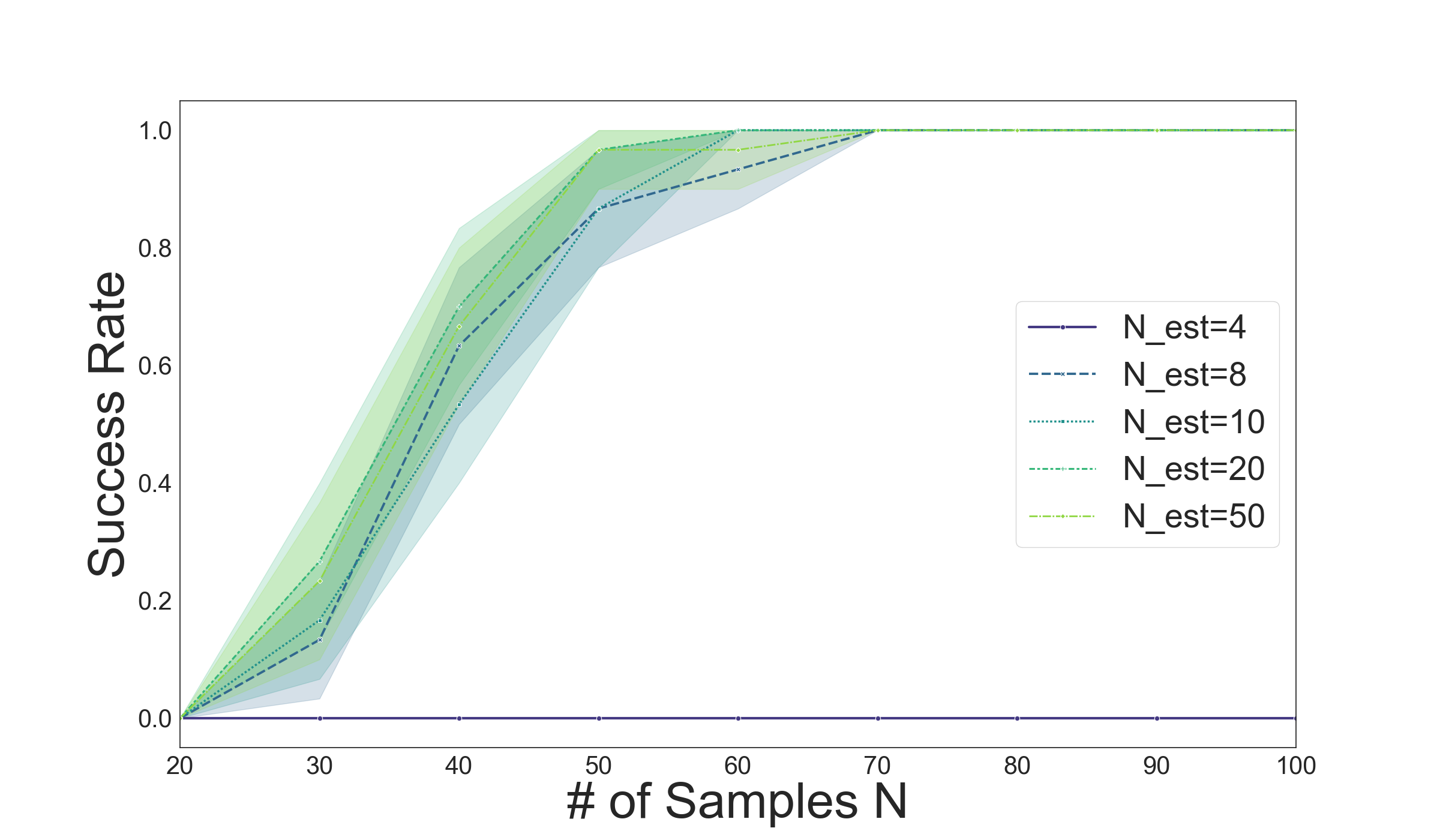}
	\caption{Probability of success vs. number of samples, for Projected-STG where we vary the number of samples for Monte Carlo estimators (``\texttt{N\_est}'') used for $\myvec{Q}$ and $\myvec{q}$ in step 1 of Algorithm \ref{algo:proj-STG}.}
	
	\label{fig:n_est_varies}
\end{figure}

{{\bf Hyperparameter tuning}:} Additionally, we studied how the number of Monte Carlo samples ($M$ in \prettyref{algo:proj-STG}) in our algorithm affects the performance.
Monte Carlo estimates are used to approximate the expectation of the gates $\vz({\bm\mu})$.
We use $D=64$ and $K=10$, and vary the number of observations $N$ from 20 to 110.
We repeat this setting with varying Monte Carlo samples from 4, 8, 10, 20, 50, and record the success rate of support recovery.
As shown in Figure \ref{fig:n_est_varies},
we found that the performances of our algorithm become indistinguishable when the number of estimates is not too small ($\geq 8$).

\subsection{Real Data}
Next, we benchmark our algorithm on two real regression datasets. We compare against the same algorithms as in the synthetic simulations (LASSO, OMP, Rand-OMP, SCAD). The two datasets are obtained from the Laboratory of Artificial Intelligence and Computer Science at the University of Porto (LIACC) and referred to respectively as AILERON -- an F-16 control problem consisting of flight data as the input variables and aileron control as the output variable, and TRIAZINES -- a quantitative structure-activity relationship regression task \cite{hirst1994quantitative} consisting of triazine structural attributes as input and molecular activity as output. To evaluate the quality of selected features, the aforementioned Success Rate is relatively uninformative as full support recovery is too difficult a bar for success. Instead, we use the True Positive Rate (TPR) and False Discovery Rate (FDR) scores, defined as
\begin{align}
	\text{TPR} &= \frac{TP}{TP + FN}\\
	\text{FDR} &= \frac{FP}{FP + TP},
\end{align}
where $\text{TP} = \text{True Positives}$, $\text{FP} = \text{False Positives}$, and $\text{FN} = \text{False Negatives}$. The error bars again denote $90\%$ confidence bounds, and we can see in Figures \ref{fig:tpr} and \ref{fig:fdr} that our algorithm equals or outperforms all competing algorithms in both measures.

\begin{figure}
	\begin{center}
		\includegraphics[width=0.45\textwidth]{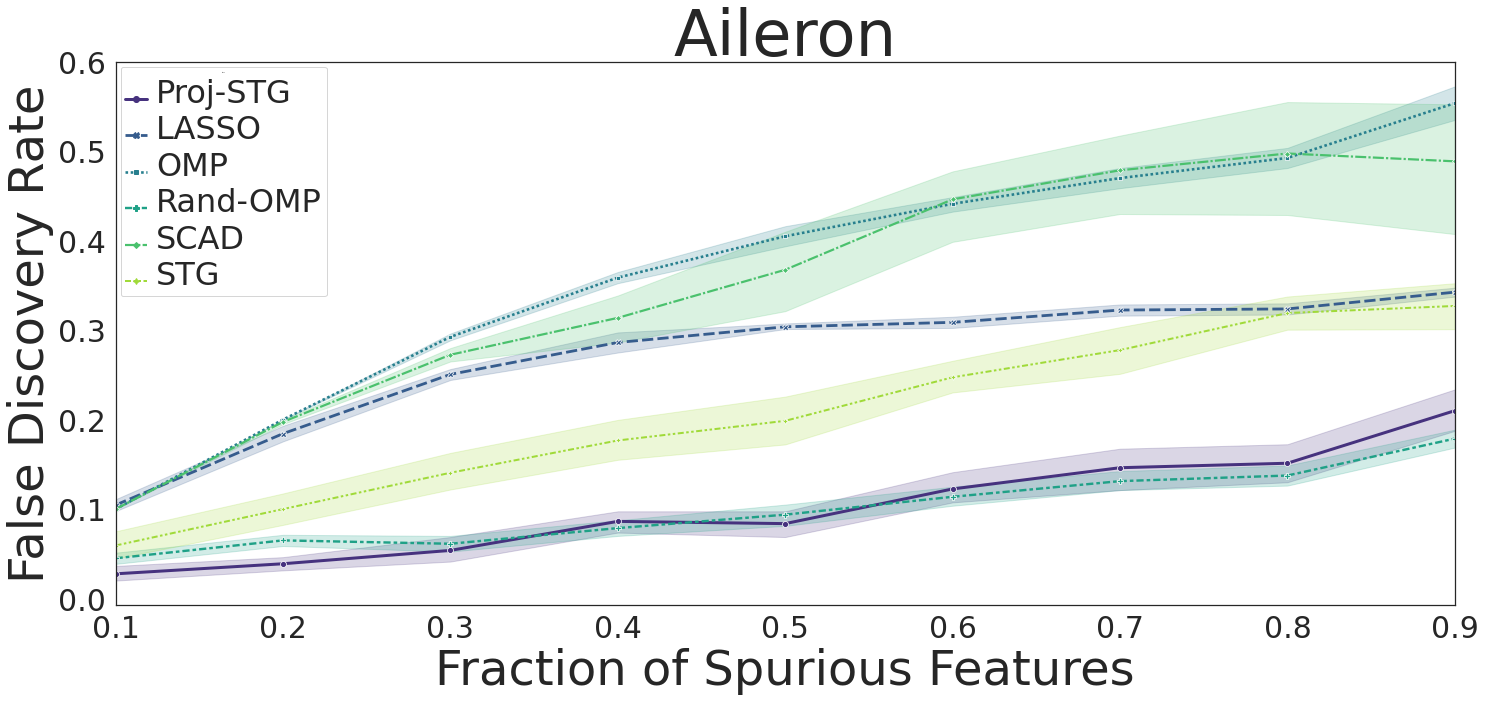}
		\includegraphics[width=0.45\textwidth]{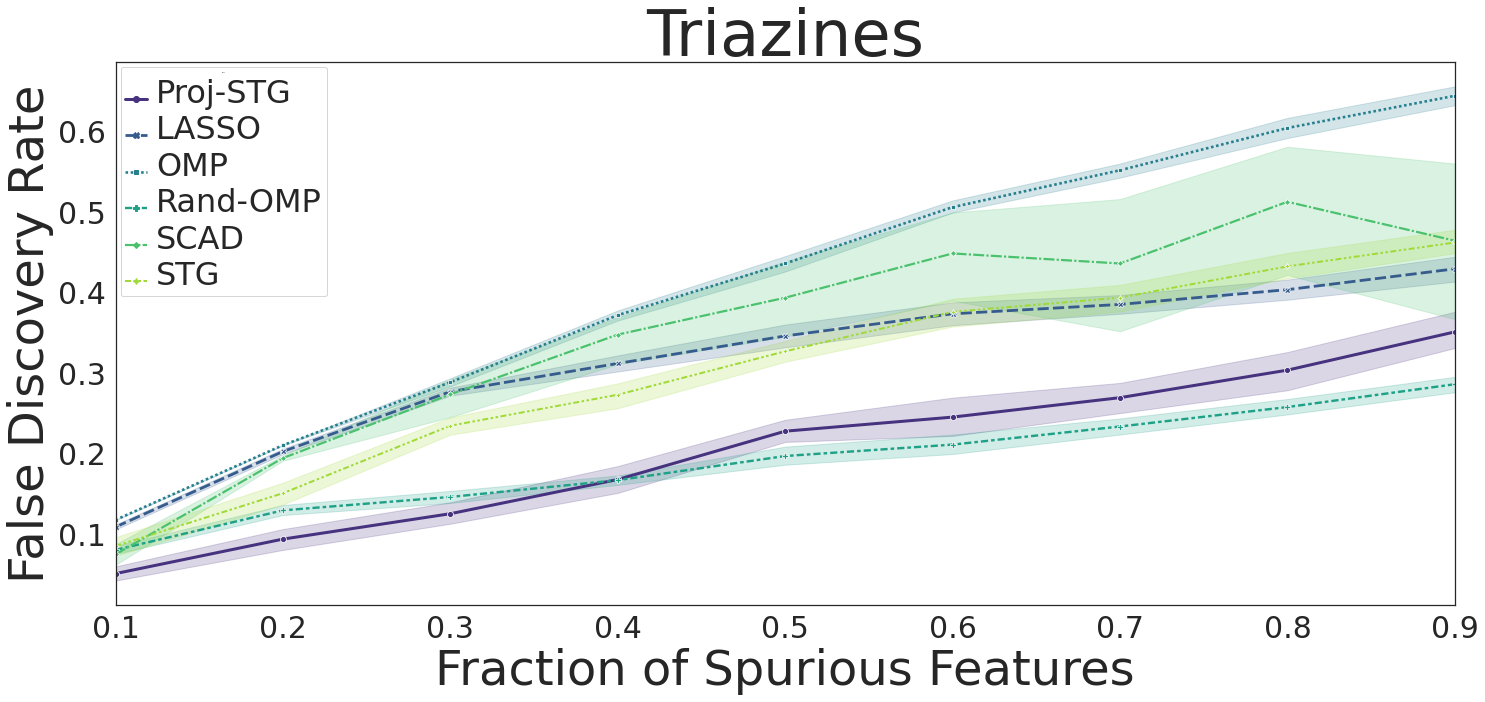}
		\caption{False Discovery Rate of Recovered Features vs. Sparsity (lower is better).}
		\label{fig:fdr}
	\end{center}
\end{figure}

\begin{figure}
	\begin{center}
		\includegraphics[width=0.45\textwidth]{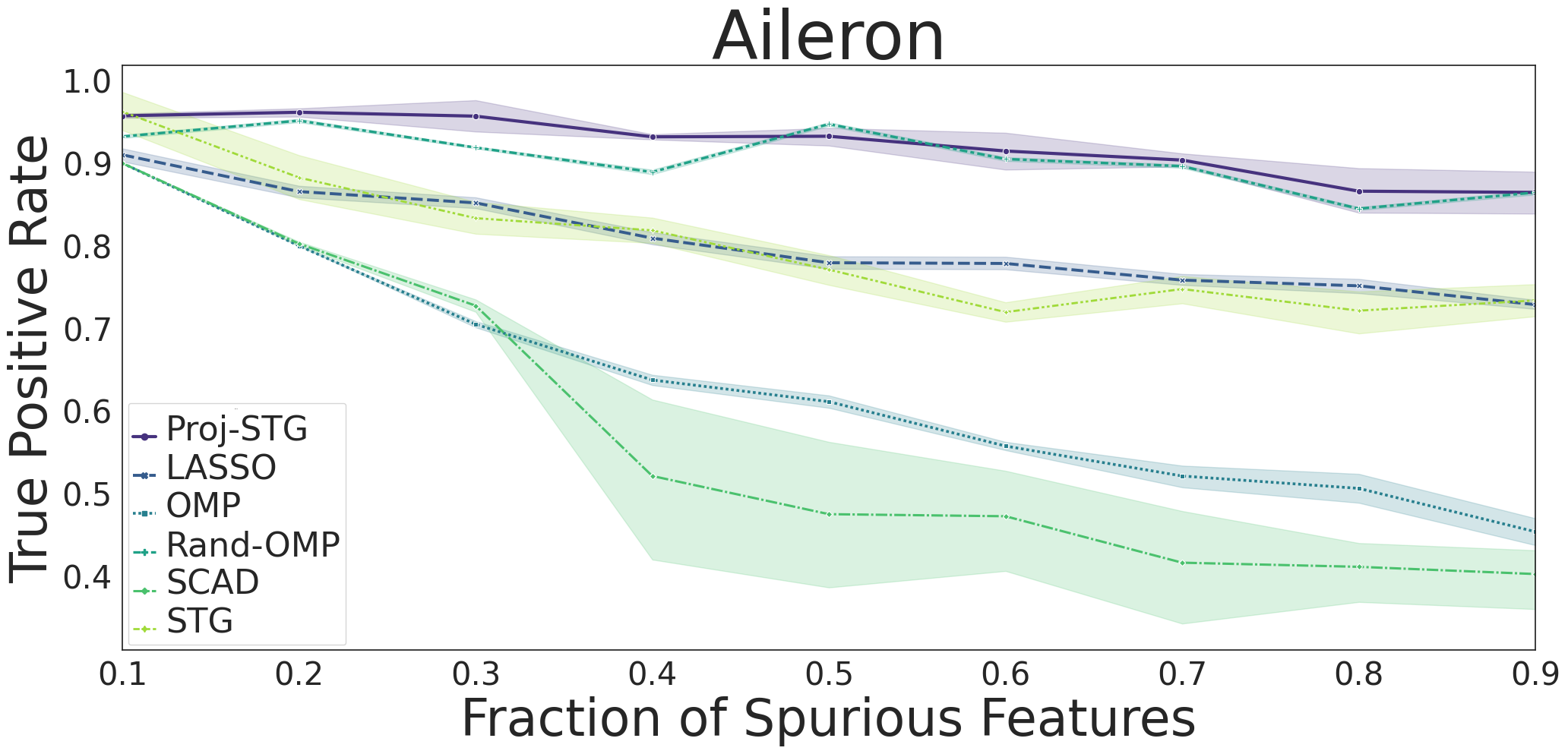}
		\includegraphics[width=0.45\textwidth]{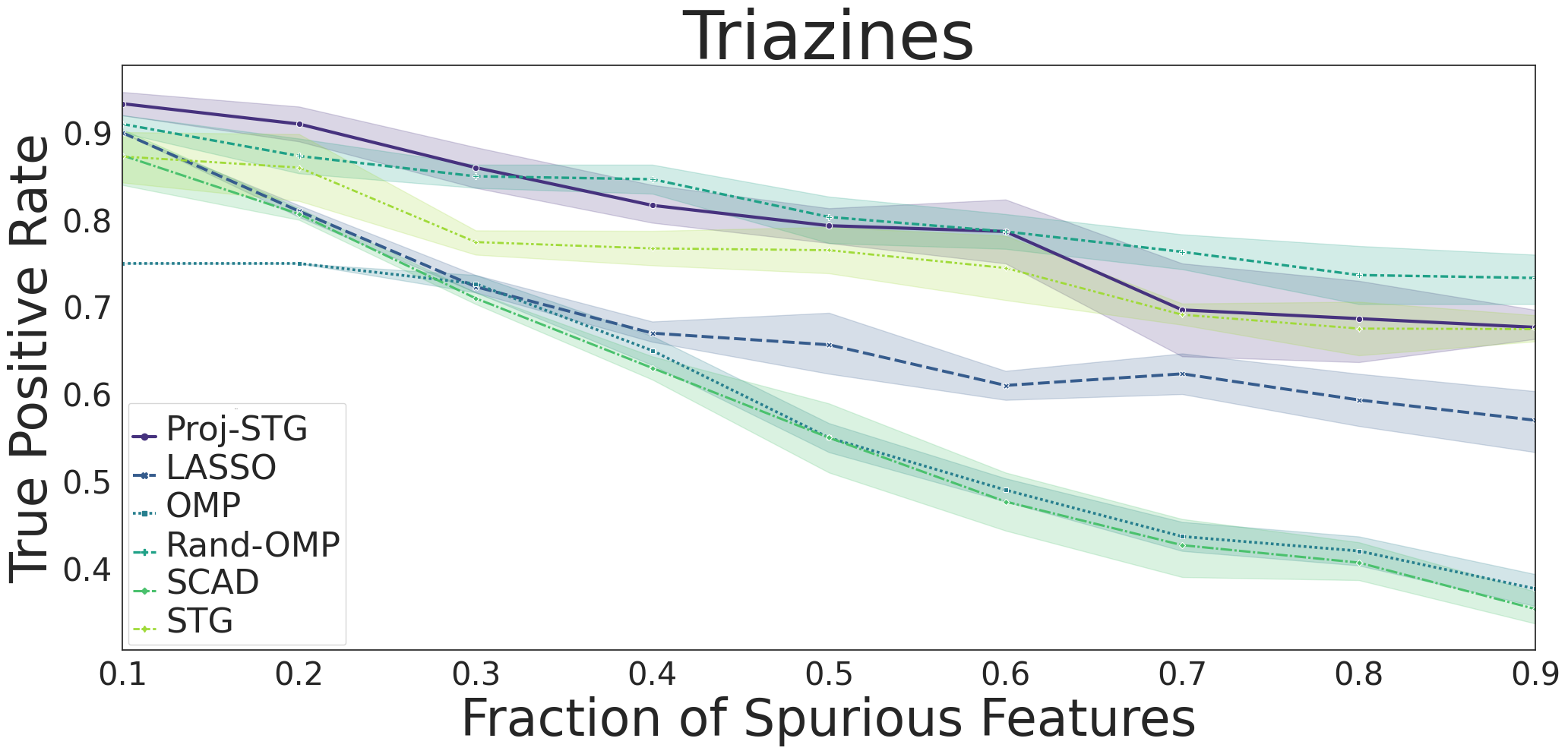}
		\caption{True Positive Rate of Recovered Features vs. Sparsity (higher is better).}
		\label{fig:tpr}
	\end{center}
\end{figure}

Our code is available on  \href{https://github.com/lihenryhfl/projection_stg}{github}.

\bibliographystyle{alpha}
\bibliography{references}

\end{document}